\documentclass[11pt, reqno]{amsart}

\usepackage{pgf,tikz}

\usepackage{amsfonts,amsmath,amsthm,amssymb}
\usepackage{graphicx}
\usepackage{epigraph}
\usepackage{verbatim}

\usepackage{url}    
\usepackage{breqn}

\usepackage{times}
\usepackage{fullpage}
\usetikzlibrary{arrows}

\newtheorem{thm}{Theorem}[section]
\newtheorem{conj}[thm]{Conjecture}

\newtheorem{lem}[thm]{Lemma}

\theoremstyle{definition}

\theoremstyle{remark}
\newtheorem{rmk}[thm]{Remark}

\newcommand{\lf}{\lfloor}
\newcommand{\rf}{\rfloor}

\newcommand{\floor}[1]{\lf {#1} \rf}

\newcommand{\NN}{{\mathbb N}}

\newcommand{\RR}{{\mathbb R}}

\newcommand{\sP}{{\mathcal P}} 
 
\newcommand{\sX}{{\mathcal X}}
\newcommand{\sY}{{\mathcal Y}}

\newcommand{\pdx}{\textrm{pdx}}

\newcommand{\Size}{{\rm Size}} 
\newcommand{\size}{{\rm size}} 
\newcommand{\Tot}{{\rm Max}}

\newcommand{\Per}{{\rm Per}}
\newcommand{\per}{{\rm per}}

\newcommand{\Sper}{W_{\rm per}}
\newcommand{\HSper}{\widehat{W}_{\rm per}}

\newcommand{\Ssize}{W_{\rm size}}
\newcommand{\SsizeX}{W^*_{{\rm size}}}
\newcommand{\HSsize}{\widehat{W}_{\rm size}}
\newcommand{\SsizeA}{W_{\rm size}^{\ast\ast}}
\newcommand{\Cper}{C_{\rm per}}
\newcommand{\HCper}{\widehat{C}_{\rm per}}
\newcommand{\CperX}{C^*_{{\rm per}}}
\newcommand{\SperX}{W^*_{{\rm per}}}

\newcommand{\Csize}{C_{\rm size}}
\newcommand{\HCsize}{\widehat{C}_{\rm size}}

\newcommand{\CsizeX}{C^*_{{\rm size}}}
\newcommand{\mult}{\rm mult}

\newcommand{\HStot}{\widehat{W}_{\rm max}}
\newcommand{\HCtot}{\widehat{C}_{\rm max}}
\newcommand{\StotX}{W^*_{{\rm max}}}
\newcommand{\CtotX}{C^*_{{\rm max}}}

\newcommand{\lp}{{\it r}}

\counterwithin*{equation}{section}
\renewcommand{\theequation}{\arabic{section}.\arabic{equation}}
\begin{document}

\noindent
\large {
\title{ Asymptotics of  reciprocal supernorm partition statistics } }

\author{Jeffrey C. Lagarias}
\thanks{The research of the first author was partially supported by NSF grant
DMS-1701576}
\address{Dept.\ of Mathematics, University of Michigan, Ann Arbor MI 48109-1043} 
\email{lagarias@umich.edu, ORCID: 0000-0002-4157-5527}

\author{Chenyang Sun}
\address{Dept.\ of Mathematics, Williams College, Williamstown, MA 01267} 
\email{quadraticreciprocity.pq@gmail.com, ORCID: 0000-0001-6755-7012}
\date{August 20, 2023, v69NA, short}

\begin{abstract} 
We consider two multiplicative statistics on the set of integer partitions: the
norm of a partition, which is the product of its parts, and the 
 supernorm of a partition, which is the product of the prime numbers $p_i$ 
 indexed by its parts $i$. We introduce and study new statistics that are sums of reciprocals of supernorms
 on three statistical ensembles of partitions, labelled by their size $|\lambda|=n$,
 their perimeter equaling $n$, and  their largest part equaling $n$. 
We show that  the cumulative statistics of the reciprocal supernorm for
 each of the  three ensembles are
asymptotic to  $e^{\gamma} \log n$ as $n \to \infty$.
\end{abstract}


\maketitle

\bibliographystyle{amsplain}

%
%
\section{Introduction}
\setcounter{equation}{0}

This paper studies asymptotics as $n \to \infty$ of  two multiplicative partition statistics, evaluated 
over all partitions in  (statistical) ensembles of integer partitions. 
To define them, let $\lambda$ denote a finite integer partition 
 in {\em part-notation} as $\lambda:= (\lambda_1, \lambda_2,  \cdots, \lambda_r)$ 
with $\lambda_1 \ge \lambda_2 \ge \cdots \ge \lambda_r \ge 1$.
Here $r=r(\lambda)$ denotes the number of its parts. We will use the
prime numbers $p_n$ indexed in increasing order, so $p_1=2, p_2=3, \cdots$.
The multiplicative partition statistics considered are built from the {\em partition norm} statistic
$N(\lambda)= \prod_{i=1}^r \lambda_i$ which is the product of its parts, and the
{\em partition supernorm statistic} $\widehat{N}(\lambda)= \prod_{i=1}^r p_{\lambda_i}$, which is the product
of the primes indexed by its parts. 

The statistics are the 
reciprocals $\frac{1}{N(\lambda)}$ and $\frac{1}{\widehat{N}(\lambda)}$
of the partition norm and partition supernorm statistics, respectively, considered for
three different  ensembles of partitions,  depending on a parameter $n$.
The first ensemble is the set $\Size(n)$ of partitions of size $n$ (the sum of its parts);
the second ensemble  is the set $\Per(n)$ of partitions of perimeter $n$ (defined in Section \ref{sec:1b}); and
(for the supernorm)  the third  ensemble is the set  $\Tot(n)$ of partitions
having largest part $n$. Definitions are given in Section \ref{sec:1b}.

Statistics  for sums of reciprocal norms of partitions of size $n$ were studied by D. H. Lehmer in 1972; his results
are  described in Section \ref{sec:15}. Lehmer showed
\begin{equation}\label{eq:size-norm-stat}
\sum_{ \size(\lambda)=n} \frac{1}{N(\lambda)} \sim e^{-\gamma}n
 \end{equation} 
 as $n \to \infty$, where  $\gamma$ is Euler's constant.

 This paper studies  the cumulative statistics of the reciprocal supernorm statistic, for these ensembles.
  The main results are stated in Section \ref{sec:2}. 
  For  the  cumulative size ensemble,  
    summed over partitions in  $\Size(k)$ for $k \le n$, we show 
 \begin{equation}\label{eq:cum-size-supernorm-stat}
 \sum_{ \size(\lambda) \le n } \frac{1}{\widehat{N}(\lambda)} \sim e^{\gamma} \log n.
 \end{equation}
 as $n \to \infty$, see Theorem \ref{thm:25}. We also  give results for reciprocal supernorms for
 the other two ensembles, 
 finding   that  the cumulative statistics for all three ensembles have   the same main term
 $e^{\gamma} \log n$.

 The reciprocal supernorm statistic is the main focus of the paper.
 In an Appendix to the paper, we  treat ensemble  statistics for the 
  reciprocal norm statistic on the $\Per(n)$
 ensemble and a modified largest part $n$ ensemble $\Tot^{\ast}(n)$ that excludes all partitions having a part of
 size $1$.  An interesting open question raised there is to  the determine the main term in the  growth rate of the reciprocal
norm statistic for the perimeter ensemble as $n \to \infty$.
 
%
 %
 \subsection{ Lehmer's asymptotics for size-ensemble  reciprocal norms} \label{sec:15}
  
 D. H. Lehmer \cite{Lehmer:72} in 1972 studied sums of reciprocal norms of 
  partitions of fixed size $n$,  which is
 \begin{equation}
 \Ssize(n) = \sum_{\lambda: \size(\lambda)=n} \frac{1}{N(\lambda)} .
 \end{equation}
 He  found the asymptotics of $\Ssize(n)$ as $n \to \infty$:
 \begin{thm}[Lehmer]\label{thm:11}
 The reciprocal norm statistic summed over partitions of size $n$ satisfies
 $$
 \Ssize(n) \sim e^{-\gamma} n 
 $$ 
 as $n \to \infty$.
 \end{thm}
 
 Lehmer proved Theorem \ref{thm:11} using generating functions; 
 he did not obtain  any estimate for the size of the error term.\\
 
 Figure \ref{fig:W-size} plots Lehmer's statistic $\Ssize (n)$ with its asymptotic $e^{-\gamma}x$,
 for $1 \le n \le 70.$
 
 
 \begin{figure}[h!]
    \centering
    \includegraphics[width=5cm]{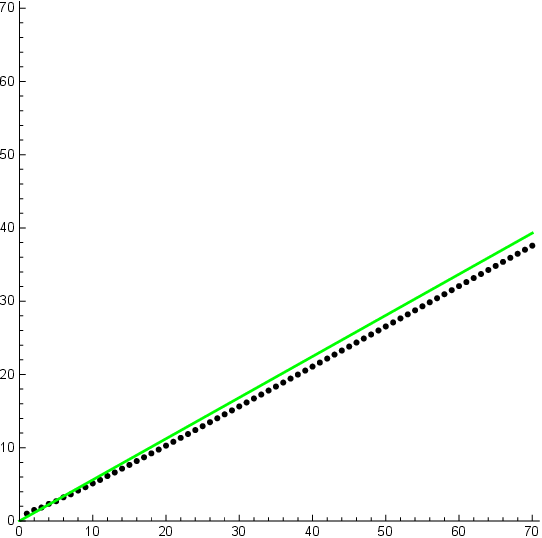}
    \caption{Plot of $W_{\text{size}}(n)$ in dots, $e^{-\gamma}x$ in green.}
    \label{fig:W-size}
\end{figure}

 Parts of size $1$ do not affect the norm  of a partition, therefore one has
 $$
\Ssize (n) = \sum_{j=0}^n \SsizeX (j),
$$ 
 where $\SsizeX(j)$ denotes the reciprocal norm statistic summed over 
 partitions of size $n$ containing no $1$'s.
Lehmer derived Theorem \ref{thm:11} via an estimate for  the individual $\SsizeX (j)$.
 
 \begin{thm}[Lehmer]\label{thm:12} 
  The reciprocal norm statistic summed over partitions of size $n$ containing no $1$'s,
 denoted $\SsizeX(n)$,   satisfies
 \begin{equation}
 \SsizeX(n) \sim e^{-\gamma} ,
 \end{equation}
 as $n \to \infty$. 
 \end{thm}
 
 Lehmer proved this result  using generating functions, and he  did not give any  estimate
 of rate of convergence.
Figure \ref{fig:W-size-1} plots Lehmer's individual statistic $\SsizeX (n)$ against  its asymptotic value $e^{-\gamma}$,
  for $1 \le n \le 40$. 


 \begin{figure}[h!]
    \centering
    \includegraphics[width=8cm]{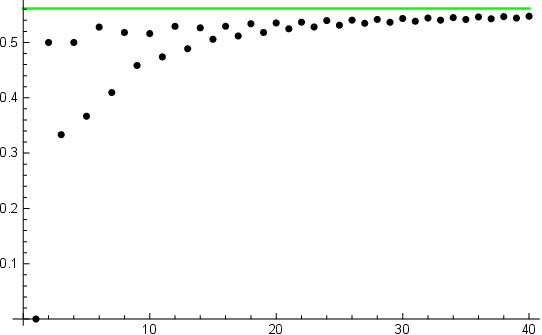}
    \caption{Plot of $\SsizeX(n)$ in dots, $e^{-\gamma}$ in green. }
    \label{fig:W-size-1}
\end{figure}

 Figure \ref{fig:W-size-1}  exhibits two lines of points for small values of $n$
 coming from odd $n$ and even $n$, with  the odd $n$ having smaller values.\medskip

 Euler's constant appears  in the limiting asymptotics for $\SsizeX(n)$,  via the constant $e^{-\gamma}$.
  
%
 %
 \subsection{ Motivation} \label{sec:15b}
The occurrence  of Euler's constant in Lehmer's work on reciprocal norms of partitions
 motivates study of analogous statistics for other multiplicative statistics of partitions.
  
  Euler's constant   is an unusual
 constant that seems unrelated to other known constants.
 It  appears in multiplicative
 number theory, prime number theory, special functions, probability theory and random matrix theory,
 see the survey paper \cite{Lagarias:13}.  
 The appearance of Euler's
 constant  in different areas has led to the discovery
 of unexpected analogies and  connections between these areas.
 In Lehmer's proof of Theorem \ref{thm:12} it enters via the constant term in the asymptotic expansion
 of the harmonic numbers $H_n = \sum_{k=1}^n \frac{1}{k}$ (which is exponentiated in the proof).

 The  supernorm statistic seems particularly worthy of study, giving a bijection between partitions
 and $\NN$ (\cite{SchS:20}), and having  compatibility with both
 additive and multiplicative orderings  on $\NN$, as  noted in \cite{Lagarias:23}.
Our results show  that Euler's constant  also appears in the asymptotic formula \eqref{eq:cum-size-supernorm-stat}
 for cumulative reciprocal supernorms of size at most $n$. This appearance 
of Euler's constant for reciprocal supernorms seems to arise  from a different mechanism, through Mertens' 
product formula,  than that 
 explaining its appearance in the asymptotic formula
 \eqref{eq:size-norm-stat} for reciprocal norms of size $n$.

%
 %
 \subsection{ Plan of paper} \label{sec:16}

\begin{enumerate}
\item
In Section \ref{sec:1b} we give definitions, notation for concepts, 
and define the three ensembles of partitions studied.
\item
In Section \ref{sec:2} we state the main results for the cumulative reciprocal  supernorm statistic,
and formulate some conjectures for the non-cumulative reciprocal supernorm statistics on
the three ensembles.
\item
In Section \ref{sec:4} we give preliminary estimates involving prime counting
functions for the supernorm estimates, needed for the proofs.

 \item
 In Section \ref{sec:5} we prove the main results for the 
 cumulative reciprocal supernorm statistic on the three ensembles.
 \item
 In Section \ref{sec:6} we raise 
  further questions, particularly that of studying  power-weighted norm
 and supernorm statistics, which are $N(\lambda)^{-\beta}$ and $\widehat{N}(\lambda)^{-\beta} $. The  parameter
 $\beta >0$  is analogous to an inverse temperature parameter in statistical mechanics. It may be
 that the value $\beta=1$ studied in this paper is of special interest.
 \item
 In Appendix \ref{sec:appendixA} we give some results on the norm statistic
 for the $\Per(n)$ ensemble and for a restricted $\Tot$ ensemble. 
 \end{enumerate}
 
%
%
\section{Partition Statistics and Ensembles}\label{sec:1b}

We recall  definitions and notation for integer partitions, formulate what are additive and multiplicative
partition statistics, and define three 
partition ensembles.
 Our notation mostly follows Stanley \cite{Stan1}, \cite{Stan2}.
 
 Let  $\sP$ denote the set of all finite integer partitions $\lambda$, including the empty partition.
We denote such a partition in {\em part-notation} as $\lambda:= (\lambda_1, \lambda_2,  \cdots, \lambda_r)$ 
with $\lambda_1 \ge \lambda_2 \ge \cdots \ge \lambda_r \ge 1$.
Here $r=r(\lambda)$ denotes the number of its parts. 

A partition  may alternatively be regarded as the multi-set of its parts, denoted   
  in {\em part-multiplicity notation} as \, 
 $\lambda = [ 1^{m_1}2^{m_2} \cdots j^{m_j} \cdots]$, where $m_j=m_j(\lambda) \ge 0$, and only finitely many 
 multiplicities $m_j$ may be nonzero.

\subsection{Additive  partition statistics}\label{sec:21} 
Additive partition statistics are statistics that add weight functions $f(i, \lambda_i)$ of individual  parts
 of a partition, which may also depend on  the indices $i$ of the numbering of the parts in the part-notation. 
 
 \begin{enumerate}
 \item[(1)]
 The    {\em length} $\ell(\lambda)$ of a partition $\lambda$    is the number of its nonzero parts,
 so  in part-notation $\ell(\lambda)=r$. The associated weight function is  $f(i, \lambda_i) = 1$ if $\lambda_i \ne 0$.
In part-multiplicity notation, $\ell(\lambda) =\sum_{j \ge 1} m_j$.
 \item[(2)]
The  {\em size} $n= |\lambda|$ of a partition $\lambda$ is the sum  of its parts. 
  $$
  n= \size (\lambda) 
  = |\lambda| := \sum_{i=1}^r \lambda_i.
  $$
  The associated weight function is $f(i, \lambda_i) = \lambda_i$.
   The empty partition $\lambda= \emptyset$ has size $0$.
 In part-multiplicity notation,  we have 
  \begin{equation}
  |\lambda| = \sum_{j \ge 1} j \cdot m_j .
  \end{equation}
 The {\em size} of a single part $\lambda_i$  of a partition $\lambda$ is its size viewed 
 as a one-part partition $\lambda_i$.
 \item[(3)]
 The {\em perimeter} $\per(\lambda)$ of a partition is
$$
\per(\lambda) :=\lambda_1 + r-1.
$$
We define  $\per(\emptyset)=1$, noting also that $\per([1]) =1$.
The associated weight function is $f(1, \lambda_1)= \lambda_1$ and, for $i \ge 2$, $f(i, \lambda_i)=1$ when $\lambda_i \ne 0$
and $f_i(i, \lambda_i)=0$ otherwise.
\item[(4)]
The {\em largest part} $\lp(\lambda)$  of a partition $\lambda$ is $\lambda_1$.
The associated weight function is $f(1, \lambda_1)= \lambda_1$ and for  $f(i, \lambda_i)=0$ for $i \ge 2$. 
The largest part of the empty partition is $0$.
 \end{enumerate}

 .

The ``perimeter" of a partition 
has a geometric interpretation in terms of the shape of
its \textit{Young diagram}. The {\em Young diagram} of a partition $\lambda$ is  a union of unit squares,
in the square lattice embedded in $\RR^2$. In the English notation, 
the Young diagram consists of a left-justified series of $r$ rows of unit squares of side $\lambda_j$ with
longest row at the top. One may picture it as having, for $1 \le j \le r$, 
with $\lambda_j$ squares in the $j$-th row, having leftmost square abutting  the $y$-axis with
lower left corner $(0, -j)$, for $1 \le j \le r.$ 
The  geometric interpretation of the  ``perimeter" in   $\per(\lambda)$  is half the length of the perimeter of the
Young diagram of a partition, minus $1$. 

A second geometric interpretation   $\per (\lambda)$ is  in terms of the 
boundary of the smallest rectangle that encloses the Young diagram of $\lambda$.
It is one less than half the perimeter of this rectangle.

A third interpretation of  $\per(\lambda)$ is that it is  the length of the largest {\em hook} of the partition:
as studied  in the paper of  Fu and Tang \cite{FuTang:18}, 
see also  Lin et al \cite{LinXY:22}.

\subsection{Multiplicative partition statistics}\label{sec:13} 

Multiplicative partition statistics
multiply arithmetic functions $f: \NN \to \RR_{>0}$
which we consider to be a weight function,
over the parts of a partition.
They are statistics of the form
\begin{equation*}
N_{f}(\lambda) = \prod_{i} f(\lambda_i).
\end{equation*} 
A multiplicative theory of integer partitions  was developed by 
Schneider, see  \cite{Sch:16}, \cite{Sch:18}.

We consider two statistics of this type: the norm and the supernorm.
\begin{enumerate}
\item[(1)]
The  {\em norm} $N (\lambda)$ of a partition, is the product of its parts
$$
N(\lambda) = \prod_{i=1}^r \lambda_i.
$$
By convention we assign $N(\emptyset)=1$. 
(The associated weight  function is $f(n) = n$.)

 \item[(2)] 
 The {\em supernorm} $\widehat{N}(\lambda)$ of a partition
 $\lambda \in \sP$   is,  for  nonempty partitions, 
 $$
 \widehat{N}(\lambda) = \prod_{i=1}^r p_{\lambda_i},
 $$
 where $p_k$ denotes the $k$-th prime in increasing order, with $p_1=2$.
 By convention $\widehat{N}(\emptyset) = 1$ and  $p_0=1$. 
 (The associated weight function  is $f(n)= p_n$, which is called
 the pre-index function in  $\pdx(n)$ in \cite{DJS:21}) .
 \end{enumerate}

 The norm  statistic has been  studied repeatedly
 under various different notations, tracing back to MacMahon \cite{MacMahon:1923} in 1923.  
The reciprocal norm, taking $f(x) = \frac{1}{x}$,
 was studied by Lehmer \cite{Lehmer:72} in 1972. 
Results on the norm of partitions are surveyed in Schneider and Sills \cite{SchS:20},
see also \cite{KumarR:20}.

 The supernorm statistic was introduced in 2021  by Dawsey, Just, and Schneider \cite{DJS:21}. 
 The paper \cite{DJS:21}  proves that  the  supernorm map $\widehat{N} : \sP \to \NN^{+}$ is  a bijection onto $\NN^{+}$.
 The  pullback from $\NN^{+}$ to $\sP$ 
of the additive total ordering on $\NN^{+}$  under this bijection $\widehat{N}$ then defines a
 total ordering $\le_{S}$ on partitions $\sP$, which we term the {\em supernorm ordering}.
 These authors note that the supernorm ordering  allows various statistics on partitions to be converted to statistics on integers,
 given as arithmetic functions. This dictionary between integer partitions and  positive integers
 provides a new optic through which to analyze various arithmetic functions. 

 The paper \cite{Lagarias:23} characterized the supernorm as the unique bijection $f: \sP \to \NN^{+}$
 in which  the image of the Young's lattice order  respects  the additive order and maps the 
multiset partial order bijectively into the divisor order. The multiset partial order is the partial order
induced by multiset inclusion, treating a partition as a multiset of positive integers.
The multiset  partial order appears  in Andrews \cite[Definition 8.2]{Andrews:88}.

 %
%
\subsection{Statistical ensembles of partitions}\label{sec:14} 
We now consider partition statistics aggregated into ensembles.  The study of such
ensembles arose in models in statistical mechanics.

The ensembles we consider  are
\begin{enumerate}
\item[(1)]
The {\em Size ensemble} $\Size (n)$, the set  of all partitions having a fixed size $|\lambda| = n$.
\item[(2)] 
The {\em Perimeter ensemble} $\Per(n)$, the set of all partitions having a
fixed perimeter \\
$\per(\lambda)=n$.
\item[(3)] 
The {\em Max-part ensemble}, $\Tot(n)$, the set of all partitions having 
largest part $\lambda_1= n$. 
\end{enumerate}

These three ensembles each give a set partition of the collection $\sP$ of all partitions.
We have 
$$\sP= \bigcup_{n \ge  0} \Size(n)=\bigcup_{n \ge 0} \Per(n)=\bigcup_{n \ge 0} \Tot(n),$$ 
where each union is disjoint.

The ensembles $\Size(n)$ and $\Per(n)$ are finite sets for each $n$, while $\Tot(n)$ is an infinite set for each $n \ge 1$.
The cardinalities of $\Size(n)$ and $\Per(n)$ are known.
\begin{enumerate}
\item[(a)]
For $\Size(n)$ Hardy and Ramanujan  \cite{HardyR:1918} showed
 \begin{equation}\label{eqn:size-growth}
 \#|\Size(n)| = p(n) \sim \frac{1}{4n \sqrt{3}} \exp \left(  \pi \sqrt{\frac{2n}{3}}  \right)  
 \end{equation}
as $n \to \infty$.  
\item[(b)]
For $\Per(n)$ one has
\begin{equation}\label{eqn:per-growth}
\#|\Per(n)|= 2^{n-1},
\end{equation} 
for all $n \ge 1$. It is proved in  \cite{FuTang:18}, Corollary 2.4, 
where it is formulated  in terms of the largest hook statistic $\Gamma(\lambda)$ in their paper.
\end{enumerate}

For the norm statistic the $\Tot$ ensemble gives an infinite answer; we must exclude partitions
have parts of size $1$; see the Appendix.

%
%
 \section{ Results}\label{sec:2} 
\setcounter{equation}{0}

 %
%
\subsection{Supernorm ensemble partition statistics}\label{sec:14a} 

This paper studies 
questions concerning the  the asymptotic behavior as $n \to \infty$ of  reciprocal supernorm partition statistics,
 attached to size, perimeter and max ensembles of partitions . For the size-ensemble,
\begin{equation}
\HSsize(n) := \sum_{\lambda:  |\lambda|=n} \frac{1}{\widehat{N}(\lambda)}.
 \end{equation}
 the perimeter-ensemble, 
 \begin{equation}
\HSper(n) := \sum_{ \lambda: \per(\lambda)=n} \frac{1}{\widehat{N}(\lambda)}.
 \end{equation}  
and the largest-part ensemble, 
\begin{equation}
\HStot(n):= \sum_{ \lambda : \lambda_1=n} \frac{1}{\widehat{N}(\lambda)}.
\end{equation}

The paper obtains rigorous results for  cumulative versions of these  statistics. These cumulative statistics are,
for the size ensemble, 
\begin{equation}
\HCsize(n) := \sum_{\lambda: |\lambda| \le n}  \frac{1}{\widehat{N}(\lambda)},
\end{equation}
the perimeter ensemble, 
\begin{equation}
\HCper(n) := \sum_{\lambda: \per(\lambda) \le n} \frac{1}{\widehat{N}(\lambda)}
 \end{equation}  
and the largest part ensemble, 
\begin{equation}
\HCtot(n) := \sum_{ \lambda: \lambda_1 \le n} \frac{1}{\widehat{N}(\lambda)}.
 \end{equation}  

We show that the asymptotic behavior as $n \to \infty$ of the  three  
cumulative quantities
$\HCsize(n)$, $\HCper(n)$ and $\HCtot(n)$ all have the {\em same} main term,
which is $e^{\gamma} \log n$, in which Euler's constant appears.
 The lower order terms of their asymptotics may differ.
Concerning the ensemble  quantities $\HSsize(n)$, $\HSper(n)$ and $\HStot(n)$, we formulate  conjectures in
Section \ref{subsec:23}. Together they lead to the conjecture made in that Section that
$\Ssize (n) \, \HSsize(n) \sim  1$ as $n \to \infty$.

  %
 %
 \subsection{Results for  reciprocal supernorms  summed over cumulative ensembles} \label{subsec:22} 
 
 In section \ref{subsec:221} we present  asymptotics for  cumulative statistics for reciprocal supernorms
 on the Size, Perimeter and Max-part ensembles.  
 In Section \ref{subsec:23} we  present conjectures for  
 individual (non-cumulative) statistics over the Size and Perimeter ensembles.

 \subsubsection{Inequalities  for supernorm statistics} \label{subsec:221}

 We have the following relations among the reciprocal supernorms  summed over cumulative
 ensembles of partitions.

 %
%
\begin{lem}
\label{lem:supernorm-ineq}
\label{lem:relations-norm}
The  cumulative reciprocal supernorm  statistics for ensembles satisfy the inequalities 
\begin{equation}
\label{eq:supernorm-ineq}
\HCsize(n) \le \HCper(n) \le \HCtot(n).
\end{equation}
These inequalities are both strict for $n \ge 4$. 
\end{lem}

These inequalities follow from set inclusions among the terms in
the three sums. The set inclusions underlying Lemma \ref{lem:supernorm-ineq}
facilitate proving results for cumulative ensembles.
We prove a result first for the Max-part cumulative ensemble, then obtain results
for other ensembles by bounding above the contributions of the supernorms
from partitions excluded from the smaller cumulative ensembles.  

In the next three subsections we give asymptotic results which together  imply 
\begin{equation}
\HCsize(n) \sim \HCper(n) \sim \HCtot(n) \sim e^{\gamma} \log n 
\end{equation}
as $n \to \infty$. We start with $\HCtot(n)$, then treat  $\HCper(n)$ and $\HCsize(n)$.

 %
 %
 \subsubsection{Asymptotics for cumulative reciprocal supernorms for the Max-part ensemble} \label{sec:224}

 \begin{thm}[Cumulative reciprocal supernorms: Max-part ensemble] 
 \label{thm:27} 
   The reciprocal supernorm statistic summed over the infinite set of partitions having  maximum part-size $\le n$ satisfies
 \begin{equation}
\HCtot (n)=e^{\gamma}(\log n+\log\log n)+\mathcal{O}\left(\frac{\log\log n}{\log n}\right),
\end{equation}
 for $n \ge 10$.
 \end{thm}
 
 Figure \ref{fig:C-hat-max} plots the statistic for $n \le 20$; the lower green line plots
the main term $f(x) = e^{\gamma}(\log x+\log\log x)$.

   \begin{figure}[h!]
    \centering
    \includegraphics{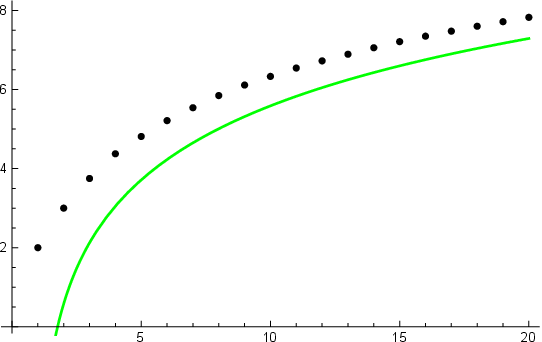}
    \caption{Plot of $\HCtot(n)$ in dots, $e^{\gamma}(\log x+\log\log x)$ in green.}
    \label{fig:C-hat-max}
\end{figure}
 
 This result  is established using the fact  that $\HCtot(n)$ is explicitly given as 
 the reciprocal of a  Mertens product. For definitions consult  also Lemma \ref{lem:mertens-prod}.

 %
 %
 \subsubsection{Asymptotics for cumulative  reciprocal supernorms for the Perimeter ensemble} \label{subsubsec:223}

 \begin{thm} [Cumulative reciprocal supernorms: Perimeter ensemble]
 \label{thm:26} 
 The reciprocal supernorm statistic summed over partitions of perimeter  $n$ satisfies 
  \begin{align}
 \label{eq:cumulative-supernorm-recip}
 \HCper(n) = e^{\gamma} (\log n + \log\log n )+ \mathcal{O}( 1),
 \end{align}
for $n \ge 10$.
 \end{thm}
 
 We will establish a result with a remainder term with explicit constants, 
 on a smaller domain, 
 in Theorem \ref{thm:52}. We think that the $\mathcal{O}(1)$ term may have
 an explicit limiting constant as $x \to \infty$, which will be negative.


  Figure \ref{fig:C-hat-per} plots $\HCper (n)$ against the bound $e^{\gamma} (\log n+ \log\log n)$ in red
  \eqref{eq:cumulative-supernorm-recip}, as well as against the less precise estimate $e^{\gamma}\log n$, in parallel with Figure \ref{fig:C-hat-size}.
     \begin{figure}[h!]
    \centering
    \includegraphics{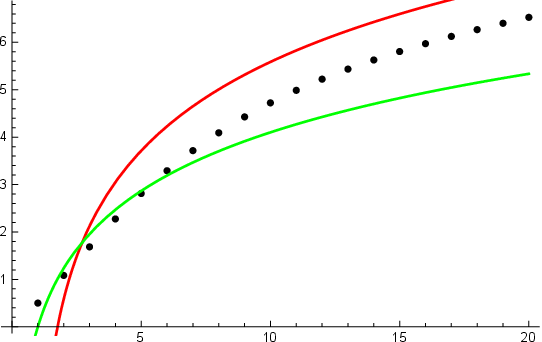}
    \caption{Plot of $\hat{C}_{\text{per}}(n)$ in dots, $e^{\gamma}(\log x+\log\log x)$ in red, $e^{\gamma}\log x$ in red.}
    \label{fig:C-hat-per}
\end{figure}

  %
 %
 \subsubsection{Asymptotics for cumulative  reciprocal supernorms for the Size ensemble} \label{subsubsec:222}

We obtain asymptotic estimates for the cumulative statistics for supernorm statistics. 
 \begin{thm}[Cumulative reciprocal supernorms: Size ensemble]
 \label{thm:25}
 The reciprocal supernorm statistic summed over partitions of size $n$ satisfies 
 \begin{equation} \label{eqn:cum-size-lower-bound}
 \HCsize(n)  \ge   e^{\gamma}\log n +  \mathcal{O}\left( 1 \right).
 \end{equation}
 for $n \ge 10$. 
  Consequently
 \begin{equation}\label{eqn:thm34}
 \HCsize(n)  =  e^{\gamma}\log n +  \mathcal{O}\left( \log\log n  \right).
 \end{equation}
 holds for $n \ge 10$. 
 \end{thm}
 
 Since $\HCtot(n) \ge \HCsize(n)$,  the upper bound in \eqref{eqn:thm34} 
 follows from  the upper bound
for $\HCtot(n)= e^{\gamma}(\log n+ \log\log n) + O(\frac{\log\log n}{\log n})$ given in Theorem \ref{thm:26}.
 We prove the lower bound  \eqref{eqn:cum-size-lower-bound} 
  in  Theorem \ref{thm:54}.  
  
   Figure \ref{fig:C-hat-size} plots $\HCsize(n)$ against the bound $e^{\gamma} (\log n)$ and a potentially more precise estimate, $e^{\gamma}(\log n+\log\log n)$, for $1 \le n \le 70$.

  \begin{figure}[h!]
    \centering
    \includegraphics{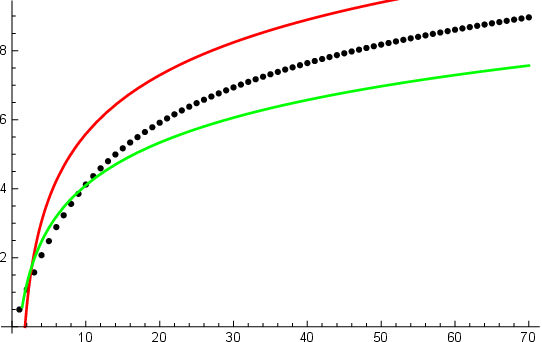}
    \caption{Plot of $\widehat{C}_{\text{size}}(n)$ in dots, $e^{\gamma}\log x$ in green, $e^{\gamma}(\log x+\log\log x)$ in red.}
    \label{fig:C-hat-size}
\end{figure}

Figure \ref{fig:C-hat-size}  suggests the possibility that  there is a lower order term
present in the asymptotics \eqref{eqn:thm34}, of order $\log\log n$.

    Figure \ref{fig:C-hat-size-loglog} plots the difference $\widehat{C}_{\text{size}}(n)-e^{\gamma}(\log n+\log\log n)$ for $n\le70$; this plot suggests the difference is bounded. If the difference is bounded, then the three statistics $\HCsize(n)$, $\HCper(n)$, $\HCtot(n)$ would  share the  main term $e^{\gamma}(\log n+\log\log n)$
 and  differ by  $O(1)$.
   

  \begin{figure}[h!]
    \centering
\includegraphics{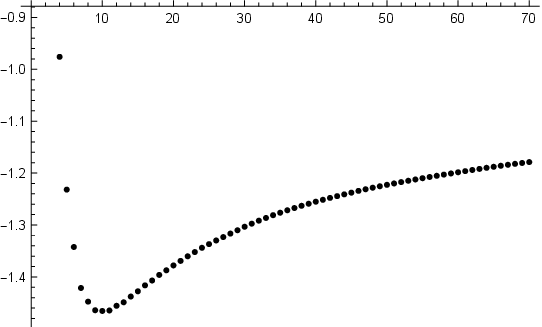}
    \caption{Plot of $\widehat{C}_{\text{size}}(n)-e^{\gamma}(\log n+\log\log n)$ for $n\le70$.}
    \label{fig:C-hat-size-loglog}
\end{figure}

 %
 %
 \subsection{Conjectures for reciprocal supernorms summed over  individual ensembles} \label{subsec:23}

 The theorems above show that  Euler's constant appears in the asymptotics of the cumulative
 sums of reciprocal supernorms, for the Size, Perimeter and Max-part ensembles,
 all of which have the same  main term $e^{\gamma} \log n$.
One can ask whether Euler's constant will also appear 
in asymptotics for the non-cumulative statistics.

The non-cumulative statistics  are given as  forward differences of their cumulative versions:
\begin{align}
   \HSsize(n)&=\HCsize(n)-\HCsize(n-1),\\
    \HSper(n)&=\HCper(n)-\HCper(n-1)\\
    \HStot(n) &= \HCtot(n)- \HCtot(n-1).
\end{align}
These non-cumulative statistics $\HSsize(n)$, $\HSper(n)$
are  however less well-behaved than the cumulative statistics,
in that  the underlying  set inclusions proving Lemma  \ref{lem:supernorm-ineq} 
 do not all hold for individual (non-cumulative) ensembles. 
 \begin{enumerate}
    \item[(i)]    
         $\Size(n)$ is not a subset of  $\Per(n)$,  for all $n \ge 4$,
because $\size(\lambda) > \per(\lambda)$ may occur. 
\item[(ii)]
 $\Per(n)$ is not a subset of  $\Tot(n)$ for all $n \ge 4$.
The set $\Tot(n)$  consists of all partitions with largest part $n$, while
the former contains some partitions with largest part smaller than $n$, for all $n \ge 4$.
\end{enumerate} 
The analogue of the leftmost inequality of \eqref{eq:supernorm-ineq}  fails at $n=14$, with  
$$
\HSsize(14) \approx 0.19381 > \HSper(14)\approx 0.19288.
$$
We do not know whether an analogue of the rightmost inequality in Lemma \ref{lem:supernorm-ineq} 
holds for non-cumulative statistics.

We  still expect 
 that the asymptotics of   the non-cumulative statistics will behave like the differenced versions
 of the asymptotics for the cumulative versions; if so, they will retain
  Euler's constant in their asymptotics. 
 Since  $\log n - \log (n-1) \sim \frac{1}{n}$ as $n \to \infty$, 
 we  formulate   two conjectures for the Size and Perimeter ensemble statistics for individual $n$.

 \begin{conj}\label{conj:supernorm-size}
 There holds 
 \begin{equation}
\HSsize(n) \sim \frac{ e^{\gamma}}{n},
 \end{equation}
 as $n \to \infty$.
 \end{conj}

 Figures \ref{fig: W-hat-size} plots the statistics $\HSsize(n)$ for sizes  $1 \le n \le 70$.
 

 \begin{figure}[h!]
    \centering
    \includegraphics{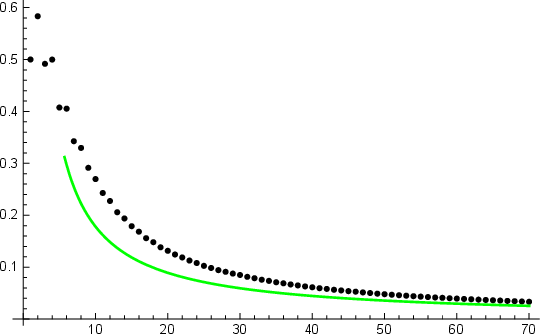}
    \caption{Plot of $\widehat{W}_{\text{size}}(n)$ in dots, $e^{\gamma}/x$ in green.}
    \label{fig: W-hat-size}
\end{figure}

\begin{conj}\label{conj:supernorm-permieter} 
There holds
 \begin{equation}
 \HSper(n) \sim \frac{ e^{\gamma}}{n},
 \end{equation} 
 as $n \to \infty$.
 \end{conj}

  Figure  \ref{fig:W-hat-per} plots the statistics $\HSper(n)$ against the asymptotic
  $f(x)= e^{\gamma}x$ (in green) for   perimeter  values $1 \le n \le 20$. \medskip

  \begin{figure}[h!]
    \centering
    \includegraphics{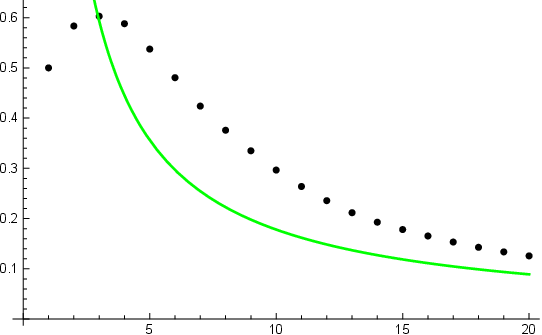}
    \caption{Plot of $\widehat{W}_{\text{per}}(n)$ in dots, $e^{\gamma}/x$ in green.}
    \label{fig:W-hat-per}
\end{figure}

We may also consider the individual  $\Tot$-ensembles.
\begin{conj}\label{conj:supernorm-total} 
There holds
 \begin{equation}
 \HStot(n) \sim \frac{ e^{\gamma}}{n},
 \end{equation} 
 as $n \to \infty$.
 \end{conj} 
 We do not present any data for $\HStot(n)$ because infinite summations are required for any finite $n$.

 Finally, if we compare Conjecture \ref{conj:supernorm-size} with Lehmer's Theorem, we
 find it can be equivalently stated as an interesting relation between  the norm and
 supernorm on the Size ensemble.
 
 \begin{conj}\label{conj:norm-supernorm}
 There holds
 \begin{equation}
\Ssize (n) \, \HSsize(n) \sim  1,
 \end{equation}
 as $n \to \infty$.
 \end{conj} 

 Euler's constant does not appear  in  this asymptotic relationship. 
 Figure \ref{fig:WWproduct} plots the product of the two statistics with the conjectured asymptotic of $1$,
 for $1 \le n \le 70.$
 

 \begin{figure}[h!]
    \centering
    \includegraphics{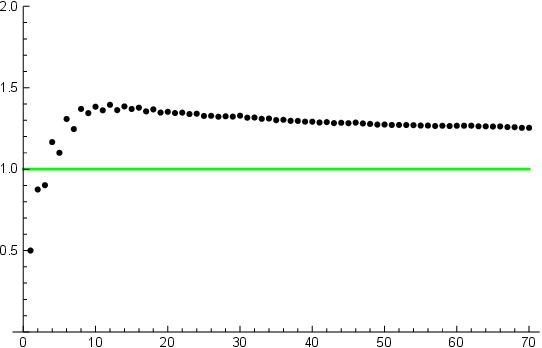}
   \caption{Plot of $W_{\text{size}}(n)\widehat{W}_{\text{size}}(n)$ in dots, $1$ in green.}
      \label{fig:WWproduct}
\end{figure}

%
%
\section{ Preliminary estimates for sums of reciprocal supernorms }\label{sec:4} 
\setcounter{equation}{0}

The proofs for the supernorm use known estimates for the size of
the $k$-th prime. We number the primes $p_j$ in increasing order, with  $p_1=2, p_2= 3, p_3=5, \cdots$.
We also use estimates  for the sum of the reciprocal primes
$R_k= \sum_{j=1}^k \frac{1}{p_j}$. 

We use the big $O$-notation for an error term to be multiplied by an unspecified positive constant,
dependent on a specified domain, 
which in the theorems in Section \ref{sec:2} is $[10, \infty)$. 
 We also use the $\mathcal{O}^{\ast}$-notation,
to indicate a remainder term with the implied  constant  being the absolute constant $1$ over the specified domain.

We start with estimates for the partial sums of reciprocal primes below $x$. It is originally due to
Mertens with the remainder term $O( \frac{1}{\log x})$, and is called Mertens' second theorem. 
\begin{lem}\label{lem:mertens2}
{\rm (Sum of reciprocal primes)} 
 There is an explicit constant $M$ (Mertens' constant) such that, for  all $x\ge 2,278,383$,
\begin{align}
\label{eqn:reciprocal-prime-sum}
    \sum_{p\le x}\frac{1}{p}=\log\log x+M+\mathcal{O}^*\left(\frac{1}{5\log^3x}\right).
\end{align}
\end{lem}
\begin{proof}
 See Dusart \cite[Theorem 5.6]{Dusart:18}, page 243. (Numerically $M\approx 0.26149$.)
\end{proof}
\begin{rmk}
It is known that the remainder term in the estimate \eqref{eqn:reciprocal-prime-sum}
in Lemma \ref{lem:mertens2} can be improved.
  The bound with  remainder term $O\left(  \exp(- C (\log x)^{1/14})\right)$  is
proved in Landau \cite[Article 55]{Landau:1909}. The power $\alpha$ of
the logarithm $O \left(  \exp(- C (\log x)^{\alpha})\right)$
be be further increased using better zero-free regions for the Riemann zeta function
obtained since that time, to at least $\alpha= \frac{1}{2}$.
\end{rmk}

We term  $P(x) = \prod_{p \le x}(1- \frac{1}{p})$ a Mertens product, 
after Mertens' work in 1874 (\cite{Mertens:1874a}, \cite{Mertens:1874b}).
see  Tenenbaum \cite[Chapter 1.6]{Tenenbaum:15}. The following result is a
improved version of
Mertens' third theorem, due to Dusart \cite{Dusart:18}. 
(Mertens obtained the weaker error term $O(\frac{1}{\log x)^2})$).
\begin{lem}
\label{lem:mertens-prod}
{\rm (Mertens products)} 

(1) 
The Mertens product satisfies, for $ x \ge 2,278,383$,
\begin{align}
  P(x)=   \prod_{p\le x}\left(1-\frac{1}{p}\right)= \frac{e^{-\gamma}}{\log x}\left(1+\mathcal{O}^*\left(\frac{1}{5\log^3x}\right)\right).\label{eqn: dus}
\end{align}

(2) The reciprocal Mertens product satisfies, for  $x\ge 2,278,383$, 
\begin{align}
\label{eq:reciprocal-Mertens} 
   1/P(x)=   \prod_{p\le x}\left(1-\frac{1}{p}\right)^{-1}= e^{\gamma}\log x\left(1+\mathcal{O}^*\left(\frac{1}{4\log ^3x}\right)\right).
\end{align}
\end{lem} 

\begin{proof}
(1) This estimate for Mertens product with the improved remainder term is given in Dusart \cite[Theorem 5.9]{Dusart:18}, page 245.

(2) Taking the reciprocal of (1) gives
\begin{align}
     e^{\gamma}\log x\left(1+\frac{1}{5\log ^3x}\right)^{-1}\le \prod_{p\le x}\left(1-\frac{1}{p}\right)^{-1}\le e^{\gamma}\log x\left(1-\frac{1}{5\log ^3x}\right)^{-1}.
\end{align}

For the bounds in \eqref{eq:reciprocal-Mertens} for $x\ge 2,278,382$,
note that for any $y\ge1$ we have 
\begin{align*}
   \left(1+ \frac{1}{5y}\right)^{-1} &\ge 1- \frac{1}{4y},\\
   \left(1-\frac{1}{5y}\right)^{-1} &\le 1+ \frac{1}{4y}.
\end{align*}
Taking $y:=\log^3x$, which is $\ge1$ for $x\ge 3$, yields the required bounds.

\end{proof} 

We use known explicit estimates for the $n$-th prime.

\begin{lem}\label{lem:primeest}
{\em ($n$-th prime estimates)} 

(1) For every $n\ge 6$,
\begin{align}
    n\log n\le p_n\le n(\log n+\log\log n).
\end{align}

(2) For every $n \ge 6$,
\begin{align}
\label{eq:corollary-axler} 
    \log n+\log\log n \le \log p_n \le \log n+\log\log n+\frac{\log\log n}{\log n}.
\end{align}

(3) For every $n\ge6$, 
\begin{align}\label{eq:corollary-axler2}
\log\log p_n=\log\log n+\frac{\log\log n}{\log n}+2\mathcal{O}^*\left( (\frac{\log\log n}{\log n})^2\right).
\end{align}
\end{lem} 

\begin{proof}
(1) This estimate appears in Axler \cite[page 2]{Axler:19}.

(2) For $n \ge 6$ the  lower bound $\log p_n \ge  \log n+\log\log n $
 is immediate taking logarithms of (1). 
The upper bound follows from
\begin{align}
    \log p_n  \le& \log (n(\log n+\log\log n)) \nonumber\\
                  =&\log n+\log\log n+\log (1+\frac{\log\log n}{\log n}) \nonumber \\
                   \le&\log n+\log\log n+\frac{\log\log n}{\log n},
\end{align}
with the last line using  $\log(1+x)\le x$ for $x \ge 0$ .

(3) Using \eqref{eq:corollary-axler} we obtain, for $n \ge 6$, 
     \begin{align}\label{eqn:firstlinear}
    \log\log p_n=&\log\left(\log n+\log\log n+\mathcal{O}^*\left(\frac{\log\log n}{\log n}\right)\right) \nonumber \\
        =&\log(\log n+\log\log n)+\log\left(1+ \frac{1}{\log n+ \log\log n } \mathcal{O}^*(\frac{\log\log n}{\log n}) \right)
         \nonumber\\
        =&\log(\log n+\log\log n)+ \mathcal{O}^*\left(\frac{\log\log n}{(\log n)^2}\right),  
         \end{align}
         with the last line using that the remainder term is nonnegative to get the constant $1$.

   We simplify the first term on the right side, for $n \ge 6$, 
    \begin{align}\label{eqn:secondlinear}
        \log (\log n+\log\log n)
         =& \log\log n+\log \left(1+\frac{\log \log n}{\log n}\right) \nonumber\\
        =& \log\log n+\frac{\log\log n}{\log n}+\mathcal{O}^*\left(\frac{\log\log n}{\log n}\right)^2,
    \end{align}
   using $\log(1+x) = x +\mathcal{O}^*(x^2)$ for $0 \le x \le 1$. 
    Substituting \eqref{eqn:secondlinear} into the right side of \eqref{eqn:firstlinear} yields \eqref{eq:corollary-axler2}. 
\end{proof}

We now obtain estimates for the reciprocal Mertens product for $x=p_n$, the $n$-th prime,
obtained by combining Lemma \ref{lem:mertens-prod} (2) with Lemma \ref{lem:primeest} (2).

\begin{lem}
\label{lem:reciprocal-prod-prime}
For $n$ such that $p_n\ge2,278,383$,
\begin{align}
  - \frac{1}{(\log n)^2}\le  \prod_{p\le p_n}\left(1-\frac{1}{p}\right)^{-1} -e^{\gamma}(\log n+\log\log n)
    \le \,\, 2\frac{\log\log n}{\log n}.
\end{align}
\end{lem}
\begin{proof}
We start from the  estimates of Lemma  \ref{lem:mertens-prod} (2)
taking  $x=p_n$, requiring  $p_n \ge 2,278,383$.
For the lower bound we weaken the lower estimate for $P(p_n)^{-1}$  to eliminate $p_n$:
\begin{align*}
    e^{\gamma}\log p_n \left(1-\frac{1}{4(\log p_n)^3}\right)\ge& e^{\gamma}(\log n+\log\log n)\left(1-\frac{1}{4(\log n)^3}\right) \nonumber\\
    \ge&e^{\gamma}(\log n+\log\log n)-\frac{ 2e^{\gamma}\log n}{4(\log n)^3}  \\
    \ge& e^{\gamma}(\log n+\log\log n)-\frac{1}{(\log n)^2}, 
\end{align*}
where we used the left side of Lemma \ref{lem:primeest} (2) in the first line.

For the upper bound we weaken  the upper estimate for $P(p_n)^{-1}$ to eliminate $p_n$:
\begin{align*}
     e^{\gamma}\log p_n \left(1+\frac{1}{4(\log p_n)^3}\right)\le&e^{\gamma}\left(\log n+\log\log n
     +\frac{\log\log n}{\log n}\right)\left(1+\frac{1}{4(\log n)^3}\right)\\ 
     \le&e^{\gamma}\left(\log n+\log\log n+\frac{\log\log n}{\log n}\right)+\frac{2e^{\gamma}\log n}{4(\log n)^3} \\
     \le&e^{\gamma}(\log n+\log\log n)+2\frac{\log\log n}{\log n}.
\end{align*}
where we applied the right side of Lemma \ref{lem:primeest} (2) in the first line, and $e^{\gamma} \le 2$  in the last line.
\end{proof}

Finally, we give an upper bound for $S(x) := \sum_{p< x}  \frac{1}{\log p}$  when $x=p_n$,
used repeatedly in the proofs.  

\begin{lem}\label{lem:logprimesum}
{\rm (Partial sum of reciprocal logarithms of primes)} 

For all $n\ge2$,
\begin{align}
  \sum_{j=1}^n\frac{1}{\log p_j}\le \frac{3n}{\log n}.
\end{align}
\end{lem}

\begin{proof}
One may check the bound directly for $2 \le n \le 4$.
For $n \ge 4$, since $p_j \ge j+1$ for $j \ge 2$, we have
\begin{align*}
    \sum_{j=1}^n \frac{1}{\log p_j} \le& \frac{1}{\log 2} + \frac{1}{\log 3} +\sum_{j=3}^{n} \frac{1}{\log j} \\
    \le& \frac{1}{\log 2} + (\floor{\sqrt{n}} -1)+ \sum_{j= \floor{\sqrt{n}}+ 1}^{n} \frac{1}{\log (\floor{\sqrt{n}}+1)} \\
    \le& \frac{1-\log 2}{\log 2} + \floor{\sqrt{n}} + (n - \floor{\sqrt{n}}) \frac{2}{\log n} \le \frac{3n}{\log n},
  \end{align*} 
as required. 
\end{proof}

%
%
\section{Proofs for Cumulative Reciprocal Supernorms}\label{sec:5} 
\setcounter{equation}{0}

We prove  results stated in Section \ref{sec:2} for  the cumulative ensemble asymptotics.  
%
%
\subsection{ Proof of Lemma  \ref{lem:supernorm-ineq}: Ensemble inequalities for reciprocal supernorms}\label{subsec:51} 

\begin{proof}[Proof of Lemma  \ref{lem:supernorm-ineq}]
(1) We have the set inclusion
$$
\{ \lambda \in \sP: \lambda \in \sP(m) \, \mbox{for some} \,\, 0 \le m \le n\}
\subseteq \{ \lambda \in \sP: \lambda \in 
\Per(n)  \, \mbox{for some} \,\, 0 \le m \le n\}
$$
The  set inclusion holds because $\size (\lambda) \ge \per(\lambda)$ holds for every partition $\lambda$.  Now\\
 $\HCsize (n)  \le \HCper(n)$
holds because $\HCsize(n)$ counts  partitions  on the left side and $\HCper(n)$ counts 
partitions on the right side.

(2) The set inclusion of all partitions with perimeter at most $n$ in the set of all partitions
having maximum part of size at most $n$ yields $\HCper(n) \le \HCtot(n)$. 

The inclusion (1) is strict for $n \ge 3$.  The inclusion (2) is strict for $n \ge 1$.
The inclusion (3) is strict for $n \ge 1$.
\end{proof}

We study 
statistics summed over cumulative ensembles because they
satisfy  set inclusions  at each level $n$, while the level $n$ non-cumulative ensembles do not satisfy set inclusions among themselves.

%
%
\subsection{ Proof of Theorem  \ref{thm:27}: Cumulative reciprocal supernorms: Max-part ensemble}\label{subsec:52} 

Theorem \ref{thm:27}  is an immediate consequence of the fact that $\HCtot(n)$ is the inverse of a  Mertens product,
which yields  a (well known)  asymptotic formula with remainder term.

\begin{thm}[asymptotic for $\HCtot(n)$]
\label{thm:C-hat-max}
We have
\begin{align}
\HCtot (n)=P(p_n)^{-1} = \prod_{j=1}^n \left(1- \frac{1}{p_j}\right)^{-1} = \prod_{j=1}^n \frac{p_j}{p_j-1}.
\end{align}
One has, for  $n\ge 2,278,333$,
\begin{align}
\HCtot (n)=e^{\gamma}(\log n+\log\log n)+\mathcal{O}^*\left(2\frac{\log\log n}{\log n}\right).
\end{align}
\end{thm}

\begin{proof}
We estimate the reciprocal supernorm sum of all partitions that have no parts greater than $n$.
 The sum is given by a product of geometric series, with the product being over all possible parts, and each series containing all possible multiplicities of each part:
\begin{align}
\label{eqn:reciprocalprimeproduct}
  \HCtot(n)=& \prod_{j-1}^n \left(1+\frac{1}{p_j}+\frac{1}{p_j^{2}}+\cdots\right) \nonumber \\
    =&\prod_{j=1}^n \left(1-\frac{1}{p_j}\right)^{-1}
    = \prod_{p\le p_n}\left(1-\frac{1}{p}\right)^{-1} .
     \end{align}
  This product is the reciprocal of the Mertens product $P(p_n) = \prod_{p \le p_n}(1- \frac{1}{p}).$
By Lemma  \ref{lem:reciprocal-prod-prime} we have, combining the upper and lower bounds, 
\begin{align}
 P(p_n)^{-1} = \prod_{p < p_n} (1- \frac{1}{p})^{-1} =    e^{\gamma}(\log n+\log\log n)+\mathcal{O}^*\left( 2 \frac{\log\log n}{\log n}\right),
\end{align}
as required.
\end{proof}

%
%
\subsection{Proof of Theorem  \ref{thm:26}: Cumulative  reciprocal supernorms: Perimeter ensemble}\label{subsec:53} 

Theorem \ref{thm:26} is an immediate consequence of the 
of the following result giving
an asymptotic formula with remainder term.


\begin{thm} [Asymptotic for $\HCper (n)$]
\label{thm:52} 
For  $n \ge 2,278,383$, there holds
\begin{align}
    \HCper(n) = e^{\gamma}(\log n+\log\log n) +\mathcal{O}^{*} \left(  4 \right).
\end{align}
\end{thm} 

Theorem \ref{thm:C-hat-max}  
gives  the upper bound for $\HCper(n)$. 
 To prove the result above it  suffices  to establish  the following  lower bound for $\HCper(n)$. 

\begin{lem}[Lower bound for $\HCper(n)$]
\label{lem:C-hat-per-lower}
For  $n \ge 2,278,383$, there holds
\begin{align}
    \HCper(n)\ge e^{\gamma}(\log n+\log\log n)-4.
\end{align}
\end{lem}

\begin{proof}[Proof of Lemma \ref{lem:C-hat-per-lower}]
For the lower bound, we first establish a lower bound for  $\HCper(7n)$ by constructing
 a large set of partitions inside $\Per ( \le 7n)$ and lower bounding their contribution to $\HCper(7n)$.
 At the end of the argument we rescale $7n$ to $n$
and show they contribute the specified amount to the sum.

We define a set of partitions $ \sY(n)$ by the two conditions:
\begin{enumerate}
    \item Condition 1: No part  of $\lambda \in \sY(n)$ is greater than $n$;
    \item Condition 2: The multiplicity   of a part $j$ in a partition $\lambda \in \sY(n)$ is no greater than $\floor{\alpha_j}$,  
    where  $(p_j)^{\alpha_j} = n^2$, i.e. taking
    \begin{align}
    \alpha_j:=\frac{2\log n}{\log p_j}.
    \end{align}
 \end{enumerate}

The proof establishes the lower bound
in three steps: (1) verifying that every specified partition in $\sY(n)$ has perimeter no greater than $7n$; (2) giving a suitable lower bound for the reciprocal-supernorm-sum over $\sY(n)$;  (3) rescaling the bound from $7n$ to $n$.\\

Step (1): We show that the perimeter of a partition  $\lambda \in \sY(n)$ 
is no greater than $7n$. We upper bound the perimeter using the  maximum possible values for the largest part and 
the number of parts:
\begin{align}
\per(\lambda) = \lp(\lambda)  + \ell(\lambda) \le  n + \sum_{j=1}^n \frac{2 \log n}{\log p_j}.
\end{align} 
Using the bound $\sum_{j=1}^n \frac{1}{\log p_j} \le \frac{3 n}{\log n}$ of Lemma \ref{lem:logprimesum}, we obtain
\begin{align}
\per(\lambda) \le n+2\log n\left(\sum_{j=1}^n\frac{1}{\log p_j} \right) \le n+ 2 \log n \left(\frac{3 n}{\log n}\right) \le 7n.
\end{align}
Thus  $\lambda \in \Per( \le 7n),$ yielding the set inclusion  $\sY(n) \subseteq \Per(\le 7n)$.
 Step 1 has thus established that 
\begin{align}
 \HCper(7n ) \ge Y(n) := \sum_{\lambda \in \sY(n)} \frac{1}{\widehat{N}(\lambda)}
 \end{align}

Step (2): 
 The reciprocal-supernorm sum $Y(n)$ taken over  partitions  $\sY(n)$ is given by the product
\begin{align}
\label{eqn:prod-ratios}
  Y(n)=&   \prod_{j=1}^n \left(1+\frac{1}{p_j}+\cdots+\frac{1}{p_j^{\floor{\alpha_j}}}\right)\nonumber\\
      = &\prod_{j=1}^n \left(1-\frac{1}{p_j^{\floor{\alpha_j}+1}}\right) \prod_{j=1}^n\left(1-\frac{1}{p_j}\right)^{-1}.
\end{align}
Noting that $p_j^{\floor{\alpha_j} + 1} \ge p_j^{\alpha_j}= n^2$, we bound the first product on
the right side  for $n \ge 2$,  by
\begin{align}
\prod_{j=1}^n \left(1-\frac{1}{p_j^{\floor{\alpha_j}+1}} \right) \ge& \prod_{i=1}^{n} \left(1- \frac{1}{n^2}\right) 
= \left(1-\frac{1}{n^2} \right)^n \ge 1- \frac{2}{n}.
\end{align} 
The second product on the right side  of \eqref{eqn:prod-ratios} is bounded below using   
the leftmost  inequality in Lemma \ref{lem:reciprocal-prod-prime}:
$$
P(p_n)^{-1} =\prod_{j=1}^n\left(1-\frac{1}{p_j}\right)^{-1} \ge e^{\gamma} (\log n + \log\log n) - \frac{1}{(\log n)^2}. 
$$ 
Combining these
bounds yields, for $p_n \ge 2, 278, 383$, 
\begin{align}
\label{eqn:Y-n-bound}
  Y(n)\ge& \left( 1- \frac{2}{n} \right) \left( e^{\gamma} (\log n + \log\log n) - \frac{1}{ (\log n)^2}\right)  \nonumber \\
\ge&  e^{\gamma} (\log n + \log\log n) + \mathcal{O}^{*} \left( \frac{1}{ (\log n)^2}  +4e^{\gamma} \frac{\log n}{n}  \right) \nonumber\\
 \ge&  e^{\gamma} (\log n + \log\log n) + \mathcal{O}^{*} \left( \frac{2}{ (\log n)^2} \right).
\end{align}

Step (3); Replacing $n$ with $n/7$ yields,   for $p_n \ge 2, 278,383$, 
\begin{align}
    \HCper(n)\ge Y(\frac{n}{7}) \ge & e^{\gamma}\left(\log\frac{n}{7}+\log\log\frac{n}7\right)
    -\mathcal{O}^{*} \left( \frac{2}{ (\log n/7)^2} \right)\nonumber \\
    \ge & e^{\gamma}(\log n+\log\log n)-4,
\end{align}
using $e^{\gamma} \log 7 \le 3.3432 < 4,$ as asserted.
\end{proof}

%
%
\subsection{Proof of Theorem  \ref{thm:25}: Cumulative reciprocal supernorms: Size ensemble}\label{subsec:54}

We prove the lower bound in Theorem \ref{thm:25}, whose statement we recall.

%
%

\begin{thm}[Lower bound for $\HCsize (n)$]
\label{thm:54}
There holds
\begin{align}\label{eqn:cum-size-explicit-bound}
    \HCsize(n)\ge e^{\gamma}\log n- O(1).
\end{align}
for $n \ge 10$.
\end{thm}

\begin{proof}
 We replace $n$ by $n^2$ and will show:
 \begin{align}\label{eqn:main-sum-2}
     \widehat{C}_{\size}(n^2)\ge 2e^{\gamma}\log n-O(1) 
 \end{align}
 holds for some range $n\ge n_0$, which implies the desired statement for $n \ge 10$, by increasing the
 $O$-constant, noting that  $\HCsize(x)$  changes  by $O(1)$ between $x=n^2$ and $x=(n+1)^2$,
 uniformly for all $n^2 \ge 9$.

Let $k=k(n)=\floor{\log n}$.  We suppose $n \ge 10$ so that $k \ge 2$. We consider the set of partitions $\sX(n^2)$  that satisfy:
\begin{equation}
\sX(n^2) = \{ \lambda = \lambda^{'} \cup \lambda^{''} : \quad \lambda \in \sX_1(n) \quad \mbox{and} \quad  \lambda^{''} \in \sX_2(n)\}, 
\end{equation}
in which $\sX_1(n)$ is the set of  partitions $\lambda^{'}$ containing only  ``small parts" $1 \le j \le \frac{n}{4}$, each
 having multiplicity no greater
than $\floor{\alpha_j}$, where
    $(p_j)^{\alpha_j} = \frac{1}{16}n^2$, i.e.,
    \begin{align}
        \alpha_j:=(2\log n- 4 \log 2)/\log p_j.
    \end{align}
and $\sX_2(n)$  is the set of  partitions $\lambda^{''}$ containing only ``large parts" $n+1 \le j \le \frac{n^2}{2k}$, 
and having   total number of parts (counted with multiplicity)  at most $k$. (The empty partition is included  in both sets.)

The partitions in $\sX_1(n)$ and $\sX_2(n)$ have no parts in common, hence the number of such partitions is
\begin{equation}
\vert \sX(n^2) \vert = \vert \sX_1(n) \vert \cdot \vert \sX_2 (n)\vert
\end{equation}

We complete the proof of \eqref{eqn:main-sum-2} in two steps: (1) verifying that every partition in $\sX(n^2)$ 
has size no greater than $n^2$, and (2) giving a suitable lower bound for the reciprocal-supernorm-sum over $\sX(n)$.\\

Step (1): We upper bound the size of a  partition $\lambda \in \sX(n)$. We have
\begin{align}
 \label{eqn: size}
   \size(\lambda)  \le &\sum_{j=1}^{\floor{n/4}}j\cdot\alpha_j+k\left(\frac{n^2}{2k}\right)  
    \le2\log n \left(\sum_{j=1}^{\floor{n/4}}\frac{j}{\log p_j}\right) +\frac{n^2}{2}\nonumber\\
    \le& 2 \log n\left(\frac{n}{4} \sum_{j=1}^{\floor{n/4}} \frac{1}{\log p_j} \right) + \frac{n^2}{2}.
    \end{align}
 We next apply the bound $\sum_{j=1}^{\floor{n/4}} \frac{1}{\log p_j} \le \frac{3 \floor{n/4}}{\log \floor{n/4}}$ of  Lemma \ref{lem:logprimesum}. 
 and combine it  with the estimate, valid for $n \ge 625$,
        \begin{align*}
            \frac{n}{4}\frac{3(n/4)}{\log \floor{n/4}}\le \frac{3n^2}{16}\frac{1}{\log n/5}\le\frac{n^2}{4\log n},
        \end{align*}
 to obtain   
    \begin{align}
    \label{eq:X-size} 
\size( \lambda)  \le 2\log n\left(\frac{n^2}{4\log n}\right)+\frac{n^2}{2} \le  n^2.
\end{align}
Thus $\lambda \in \sP(\le n^2)$, yielding the set inclusion  $\sX(n) \subseteq \sP(\le n^2)$. 

We obtain  the bound
\begin{align}
\label{eqn:C-X-bound}
 \HCsize(n^2)= \sum_{\size(\lambda) \le n^2} \frac{1}{\widehat{N}(\lambda)} \ge \widehat{X}(n^2) := \sum_{\lambda \in \sX(n)} \frac{1}{\widehat{N}(\lambda)}.
 \end{align}

Step (2): 
The supernorm of a multiset union  $\lambda = \lambda^{'} \cup \lambda^{''} \in \sX(n^2)$ satisfies 
\begin{equation}\label{eqn:hat-N-product}
\widehat{N}(\lambda) = \widehat{N}(\lambda^{'}) \widehat{N}(\lambda^{''}). 
\end{equation} 
The product decomposition of $\sX(n^2)$  gives 
\begin{align}
\label{eqn:two-part}
  \widehat{X}(n^2) = \widehat{X}_1(n) \widehat{X}_2(n) := S_{small}(n) \cdot S_{large}(n), 
\end{align}
 where $S_{small}(n)$ is the sum of reciprocal supernorms taken over $\sX_1(n)$ and
 $S_{large}(n)$ is the sum of reciprocal supernorms taken over $\sX_2(n)$.
   We lower bound  these two sums  separately.\\

The sum $S_{small}(n)$  is given by the product
\begin{align}
    S_{small}(n)=&\prod_{j=1}^{n/4}\left(1+\frac{1}{p_j}+\cdots+\frac{1}{p_j^{\floor{\alpha_j}}}\right)=Y(\frac{n}{4}),
    \end{align}
where $Y(n)$ is given in  \eqref{eqn:prod-ratios}.
This product  is bounded below in \eqref{eqn:Y-n-bound}, for $n\ge 4 \cdot  2,278,383$, by:
\begin{align}
S_{small}(n) =  Y(\frac{n}{4})  =&e^{\gamma}\left(\log\frac{n}{4}+\log\log\frac{n}{4}\right)+\mathcal{O}^*\left(\frac{2}{(\log n/4)^2}\right) \nonumber \\
    \ge& e^{\gamma}(\log n+\log\log n)-3.
\end{align}
The inequality between the last line  and right side of previous line holds for $n\ge 520$.

The sum $S_{large}(n)$ is the sum of reciprocal supernorms of partitions in $\sX_2(n)$.
 We will  show that for $n\ge 10$, that 
\begin{align}
\label{eqn:S-large-bound}
S_{large}(n) \ge 2-2\frac{\log\log n}{\log n}-O ( \frac{1}{\log n}). 
\end{align}
Assuming \eqref{eqn:S-large-bound} is verified,  we may combine the sums $S_{small}(n)$ and $S_{large}(n)$ and get
\begin{align}
    \HCsize(n^2)\ge &S_{small}(n)S_{large}(n) \nonumber \\
    \ge&\left(e^{\gamma}\left(\log n+\log\log n\right)-3\right)\left(2-2\frac{\log\log n}{\log n}- O(\frac{1}{\log n})\right) 
         \nonumber \\
    \ge &2e^{\gamma}\log n- O \left( 1 \right).
\end{align}
which is the desired lower bound, and  it holds for $n \ge 10$, adjusting the $O$-constant.


It remains to establish \eqref{eqn:S-large-bound}.
We start from
\begin{align}
\label{eqn:supernormsum}
 S_{large}(n)= 1+S_{n;k}(1)+S_{n;k}(2)+\cdots+S_{n;k}(k),\end{align}
where $S_{n;k}(j)$ is the sum of $1/\hat{N}(\lambda'')$ taken over those $\lambda'' \in \sX_2(n)$ having
exactly  $j$ parts. 
We have the bound
\begin{align}\label{eqn:S-large-loss-bound}
    S_{n;k}(j)\ge \frac{1}{j!} \left(\frac{1}{p_{n+1}}+\frac{1}{p_{n+2}}+\cdots+\frac{1}{p_{\floor{n^2/(2k)}}}\right)^j ,
\end{align}
since the $j$th-power sum, when expanded, generates no more than $j!$ copies of every term in $S_{n;k}(j)$. 
Substituting this bound in \eqref{eqn:supernormsum} yields 
\begin{align}
\label{eqn:Slarge-bound}
S_{large}(n) \ge& \sum_{j=0}^k \frac{1}{j!} (z_n)^k
\end{align}
where we set  
\begin{align}
z_n  : =  \sum_{ n+1 \le j \le \floor{n^2/2k}} \frac{1}{p_j},
\end{align}
noting that $k = \floor {\log n}$ is determined by $n$. 
Our next goal is to estimate $z_n$.
We start with 
\begin{align}
\label{eqn:diffloglog}
z_n= \sum_{ n+1 \le j \le \floor{n^2/2k}} \frac{1}{p_j}     
 =& \log\log p_{\floor{n^2/(2k)}}-\log\log p_{n}+\mathcal{O}^*\left(\frac{1}{2(\log n)^3}\right).
    \end{align}
where the rightmost equality used  Lemma \ref{lem:mertens2} with $x= p_{\floor{n^2/2k}}$ and with  $x= p_{n+1}$,
requiring $\floor{n^2/2k}\ge n+1 \ge 2,278,383$.

By Lemma \ref{lem:primeest} (3), we have 
\begin{align}\label{eq:nasty}
    \log\log p_{\floor{n^2/(2k)}}=&\log\log\floor{\frac{n^2}{2k}}+\frac{\log\log \floor{n^2/(2k)}}{\log\floor{n^2/(2k)}}
    +2\mathcal{O}^*\left((\frac{\log\log n}{\log n})^2\right).
  \end{align}
After 
using $|\floor{\frac{n^2}{2k}} - \frac{n^2}{2 \log n}| = O \left( \frac{n^2}{(\log n)^2}\right)$,
we obtain 
\begin{align}\label{eq:s0}
  \log\log p_{\floor{n^2/(2k)}}   =& \log\log \frac{n^2}{2\log n}+\frac{\log\log (n^2/2\log n)}{\log (n^2/2\log n)}
    +\mathcal{O}\left(\left(\frac{\log\log n}{\log n}\right)^2\right).
\end{align}
Algebraic simplification of the right side yields
\begin{align}
    \log\log p_{\floor{n^2/(2k)}}=\log\log n+\log 2+  O\left(\frac{1}{\log n}\right).
\end{align}
(The $\frac{\log\log n}{\log n}$ terms visible in \eqref{eq:nasty} cancel out.)
Substituting  into  \eqref{eqn:diffloglog} this bound along with  the estimate of Lemma \ref{lem:primeest} (3) for $\log\log p_n$  yields
for $n \ge 2,278,383$, 
\begin{align} 
\label{eqn:diffloglog2}
 z_n=  \log 2- \frac{\log\log n}{\log n}+O \left( \frac{1}{\log n}\right).
\end{align} 
This bound implies  $0< z_n < 1$ holds for sufficiently large $n\ge n_0$, say, and
an (omitted) error analysis shows it holds for $n \ge 4 \cdot 2,278,383.$
To conclude the estimate,  \eqref{eqn:Slarge-bound} now yields  
\begin{align}
\label{eqn:large-approx}
   S_{large}(n) \ge&\sum_{k=0}^k \frac{1}{j!} (z_n)^j 
\ge \exp( z_n) - \sum_{j= k+1}^{\infty} \frac{z^j}{j!}
\ge \exp(z_n) - \frac{1}{k!}, 
\end{align}
Now \eqref{eqn:diffloglog2} gives 
 \begin{align}
 \exp(z_n) \ge& 2 \exp\left(- \frac{\log\log n}{\log n} -O\left( \frac{1}{\log n} \right)\right)
  \ge 2\left(1- \frac{\log\log n}{\log n} - O\left( \frac{1}{\log n}\right) \right) 
 \end{align} 
 for sufficiently large $n \ge n_0$. 

Substituting this bound in \eqref{eqn:large-approx} and using $\frac{1}{k!} = \mathcal{O}^*\left(\frac{1}{\log n}\right)$
for $n \ge 21$, we obtain 
\begin{align} 
 S_{large}(n)
  \ge& 2-2\frac{\log\log n}{\log n} - O \left( \frac{1}{\log n}\right).
\end{align}
for $n \ge 4 \cdot 2,278, 383$. Increasing the $O$-constant makes it hold for $n \ge 10.$
 This verifies   \eqref{eqn:S-large-bound}, as required. 
\end{proof}

\begin{rmk}
The discussion in Section \ref{subsubsec:222} suggests 
that the right side of \eqref{eqn:cum-size-explicit-bound} may be improvable by a term  $\log\log n$. 
A place to look for further savings in the proof  presented  is the  non-trivial loss 
present  in the inequality \eqref{eqn:S-large-loss-bound}.  
\end{rmk}

%
\section{Concluding Remarks}\label{sec:6} 
\setcounter{equation}{0}

We raise some problems for further work.

\subsection{Temperature variable: $\beta$-parametrized multiplicative statistics} 
All of the norm and supernorm partition statistics of this paper can be generalized to 
allow the weights to be arbitrary (real) powers:
\begin{align}
\label{eq:beta-norm} 
N_{\beta}(\lambda) := \left(\prod_{i} \lambda_i \right)^{-\beta} 
\end{align}
and
\begin{align}
\label{eq:beta-supernorm}
\widehat{N}_{\beta} (\lambda) := \left(\prod_{i}  p_{\lambda_i} \right)^{- \beta}
\end{align}

What is the asymptotic behavior of these statistics for fixed  $0\le \beta < \infty$
for the norm and for the supernorm, on the three partition
ensembles studied in this paper?

In analogy with statistical mechanics, the parameter $\beta$ may be
viewed as an extensive variable analogous to (reciprocal) temperature (when $\beta>0$), 
This paper treated the case $\beta=1$. Conjecture \ref{conj:norm-supernorm} suggests  that something
``special" is happening at the parameter value $\beta=1$, between the norm and the supernorm. 
Study of the statistics for other values of $\beta$ could add perspective
to this observation.
For $\beta=0$ the statistic 
measures  the cardinality of the ensemble. It is the infinite temperature  (``gas") regime.
The appearance of Euler's constant is likely specific to the value $\beta=1$.
%
%
\subsection{Supernorm-weighted partitions with distinct parts}
 In  his 1972 work Lehmer \cite{Lehmer:72}  also  obtained asymptotics for  norm-weighted partitions with 
distinct parts. 
.
  
 \begin{thm}[Lehmer]\label{thm:3} 
 The reciprocal norm statistic summed over partitions of size $n$ with distinct parts satisfies
 $$
 \SsizeA(n) \sim e^{-\gamma} ,
 $$
 as $n \to \infty$.
 \end{thm}
 
  What are the analogous asymptotics for partitions with distinct parts when the norm is replaced by the supernorm? 
 In the latter case, the answer must be scaled by  some function  of $n$ which goes
 to zero as $n \to \infty$. Will Euler's constant appear in the answer?


\subsection{Prime indexing function} 

 The supernorm statistic is obtained from the norm statistic multiplicatively through the 
application term by term of the {\em  prime-indexing map} $\pdx: \NN \to \mathbb{P}$, 
where $\mathbb{P}$ denotes the
set of prime numbers, 
given by $\pdx(n)=p_n$,
the $n$-th prime, for $n \ge 1$, with  $\pdx(0)=1$. (The map $\pdx$ is called the pre-index in \cite{DJS:21}.) 
Can one uncover new information on the behavior of the arithmetic function $\pdx$
through study  of asymptotics of other functions of the norm  statistics and
supernorm statistics, and 
determination of   lower order terms
in the asymptotics?  
 %
%
\section*{ Acknowledgments.} We thank Robert Schneider and Robert Vaughan for  helpful references.
Work of the first author was partially supported by NSF grant DMS-1701576.
Work of the  second author was supported by the Roche and Gomez Research Fellowship from Williams College. 

 %
%

\section*{Declarations.}

 %
%

\appendix
\counterwithin*{equation}{section}

%
%
\section{Reciprocal norm statistics}\label{sec:appendixA}
 \renewcommand\theequation{A.\arabic{equation}}

 This appendix formulates  and proves inequalities bounding  reciprocal norm statistics for the
 three partition ensembles $\Size(n), \Per(n)$ and $\Tot^{\ast}(n)$.
  
 We define  reciprocal norm partition statistics attached to the three  ensembles.
 For the size ensemble, 
\begin{equation}
\Ssize(n) := \sum_{\lambda:  |\lambda|=n} \frac{1}{N(\lambda)}.
 \end{equation}
 and for the perimeter ensemble, 
 \begin{equation}
\Sper(n) := \sum_{ \lambda: \per(\lambda)=n} \frac{1}{N(\lambda)}.
 \end{equation}  
 We also define  cumulative versions of these  statistics. The cumulative statistics
 for the size ensemble  are
\begin{equation}
\Csize(n) := \sum_{\lambda: \size(\lambda) \le n}  \frac{1}{N(\lambda)},
\end{equation}
and for the perimeter ensemble are 
\begin{equation}
\Cper(n) := \sum_{\lambda: \per(\lambda) \le n} \frac{1}{N(\lambda)}
 \end{equation}
All individual and cumulative sums of reciprocal norms 
over the largest part ensemble diverge for $n \ge 2$, due to the presence
of all partitions $\lambda= 1^k$ for all $k \ge 1$.
   
   The norm statistic is not affected by the parts of size $1$ in the partition $\lambda$,
   and the $\Tot$ ensemble gives infinite values for the reciprocal norm statistic.
is convenient to introduce the  $1$-removed reciprocal norm statistics  for the size ensemble,  
  \begin{equation}
\CsizeX(n) := \sum_{\lambda:\,  |\lambda| \le n, \,1 \not\in \lambda} \frac{1}{N(\lambda)},
\end{equation}  
the perimeter ensemble, 
\begin{equation}
\CperX(n) := \sum_{\lambda: \,\per(\lambda) \le n, 1 \not\in \lambda} \frac{1}{N(\lambda)}. 
\end{equation}  
and the largest part ensemble,
\begin{equation}
\CtotX(n) := \sum_{\lambda: \,\lambda_1 \le n, 1 \not\in \lambda} \frac{1}{N(\lambda)}. 
\end{equation}
where the last statistic  is now finite for all $n$.

These statistics correspond to counts over
members of restricted partition ensembles that omit partitions having a partition of size $1$.
\begin{enumerate}
\item[(1)]
The {\em restricted Size ensemble} $\Size^{\ast} (n)$, the set  of all partitions having a fixed size $|\lambda| = n$
and no part of size $1$.
\item[(2)] 
The {\em restricted Perimeter ensemble} $\Per^{\ast}(n)$, the set of all partitions having a
fixed perimeter 
$\per(\lambda)=n$, and no part of size $1$.
\item[(3)] 
The {\em restricted Max-part ensemble}, $\Tot^{\ast}(n)$, the set of all partitions having 
largest part $\lambda_1= n$, and no part of size $1$.
\end{enumerate}  
 The sets  $\Size^{\ast}(n)$, $\Per^{\ast}(n)$ are finite sets, while
 $\Tot^{\ast}(n)$ is an  infinite set for all $n \ge 2$.

 %
 %
  \subsection{Inequalities for reciprocal norm statistics} \label{sec:211}

 We have the following  relations among the reciprocal norms summed over different sets of partitions.
 
 %
%

\begin{lem}
\label{lem:norm-ineq}
The norm statistics satisfy the inequalities 
\begin{equation}
\Ssize(n) \le \Sper(n) =  \CperX(n) \le \CtotX(n).
\end{equation}
We also have
\begin{equation}\label{eqn:A11}
\Ssize(n) = \CsizeX(n).
\end{equation} 
\end{lem}
 The first inequality holds by construction of a norm-preserving injection from the set of size-$n$ partitions to the set of perimeter-$n$ partitions. 
 The second inequality holds by set inclusion of the partitions of perimeter up to $n$ containing no $1$'s in the set of partitions with
 maximal part up to $n$ containing no $1$'s. 
 The third equality  \eqref{eqn:A11} is an identity obtained by removing parts of size $1$.

 %
 %
 \subsubsection{Bounds for perimeter ensemble reciprocal norms} \label{sec:212}

We show an upper bound for the reciprocal norms summed over the  Perimeter ensemble.

 \begin{thm}[Perimeter ensemble reciprocal norms]\label{thm:W-per} 
 The reciprocal norm statistic summed over partitions of perimeter  $n$ satisfies, for all $n \ge 1$,
 \begin{equation}
 \label{eq:per-upper}
 \Sper(n) \le \  n.
 \end{equation} 
 \end{thm}
 
 This result is proved by a telescoping product upper bound. 
 (It also  follows  from  the 
  exact formula  for $\CtotX(n)$ given in Theorem \ref{thm:W-tot} below.)
 Since $\Sper(n) \ge \Ssize(n)$ by Lemma \ref{lem:norm-ineq}, Lehmer's asymptotic bound  for $\Ssize(n)$ implies
 $$
 \Sper(n) \ge e^{-\gamma} n + o(n)
 $$
 as $n \to \infty$.


  \begin{figure}[h!]
    \centering
    \includegraphics[width=6cm]{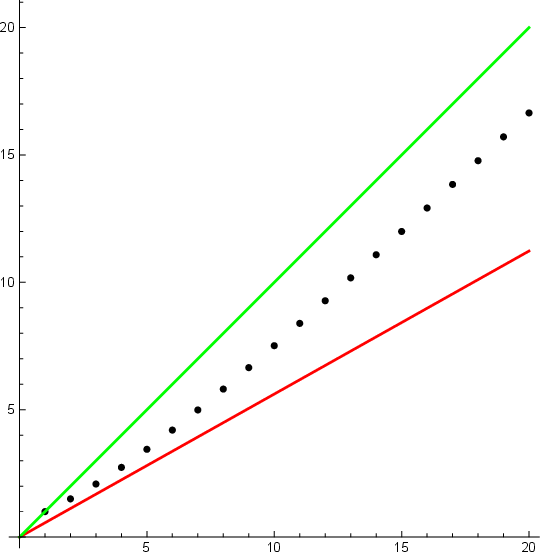}
    \caption{Plot of $W_{\text{per}}(n)$ in dots, $x$ in green, $e^{-\gamma}x$ in red.}
    \label{fig:W-per}
\end{figure}

   Figure \ref{fig:W-per} plots $\Sper(n)$ for $1 \le n \le 20$
   against the upper bound $n$ (from \eqref{eq:per-upper})
    and against Lehmer's asymptotic  lower bound $e^{-\gamma}n$. 
We can define quantities $\alpha_0$ and $\alpha_1$ by
 $$
 \alpha_0 := \liminf \frac{\Sper(n)}{n} \le \limsup \frac{\Sper (n)}{n} =: \alpha_1,
 $$
 whence $e^{-\gamma} \le \alpha_0 \le \alpha_1 \le 1$. 
 The trend of  Figure \ref{fig:W-per} suggests the possibility that $\alpha_0=\alpha_1=1$.
 
 
 \subsubsection{Bounds for maximum-part ensemble  reciprocal norms} \label{sec:213}

 We obtain closed form answers for $\StotX(n)$ which yield an exact answer for $\CtotX(n)$.

 \begin{thm}[Total reciprocal norms]\label{thm:W-tot} 
 The reciprocal norm statistic for partitions with largest part $n$ and  no parts of size $1$,
 which is $\StotX(n)$, satisfies $\StotX(1)=0$ and, 
 for $n=0$ and $n \ge 2$, 
  \begin{equation}\label{eqn:max-part-norm-sum}
 \StotX( n)=1.
 \end{equation}
 Consequently the reciprocal norm statistic summed over the infinite set of partitions with all parts at most  $n$,
 and no parts of size $1$,  satisfies, for $n \ge 1$, 
 \begin{equation} \label{eqn:cum-max-part-norm} 
 \CtotX( n)=n.
 \end{equation}
 \end{thm}

 The bound of  Theorem \ref{thm:W-per}
 for perimeter ensemble reciprocal norm statistics  follows from 
 \eqref{eqn:cum-max-part-norm} via Lemma \ref{lem:norm-ineq}(1).

%
%
\subsection{Reciprocal norm statistics: Proofs}\label{sec:appendixB}

We prove the results stated in Section \ref{sec:211}.
%
%
\subsubsection{ Proof of Lemma  \ref{lem:norm-ineq}: Inequalities among ensembles of  reciprocal norms}\label{subsec:31} 

\begin{proof}[Proof of Lemma \ref{lem:norm-ineq}] 
(1) The subtlety in  showing  $\Ssize(n) \le \Sper(n)$ is that partitions of
size $n$ may have perimeter strictly smaller than $n$. We prove it by constructing  a norm-preserving injection from
the set of partitions $\sP(n)$ of size $n$ into the set of partitions $\Per(n)$ of perimeter $n$. 
Each  $\lambda^{'} \in \sP(n)$ has
$n = \size(\lambda') \ge \per(\lambda')$. Given $\lambda'$ form  a new partition
$\lambda := \lambda^{'} \cup \{ 1^{\size(\lambda')- \per(\lambda')}\}$ ( a multiset union).
This $\lambda$ has perimeter $\per(\lambda) = n$, because adding new parts of size  $1$ to a partition 
increases the perimeter by the number of added $1$'s. The map $\lambda' \to \lambda$
leaves the norm invariant: $N(\lambda) = N(\lambda')$. 
It  is an injective map  $\sP(n)$ into  $\Per(n)$ 
because $\lambda'$ is uniquely recoverable from $\lambda$, using the identity
 $\size(\lambda) = 2n- \per(\lambda')$ to determine how many $1$'s to remove from $\lambda$.
We conclude   $\Ssize(n) \le \Sper(n)$.

(2) To show $\Sper(n) =\CperX(n)$ we give a bijective norm-preserving
map between the two sets of partitions used in the sums.
 We again use the fact that parts of size $1$ do not change the norm. Each partition $\lambda$
of perimeter $n$, has an associated partition $\lambda'$ of some perimeter  $0\le k \le n$ obtained by dropping all the $1$'s in $\lambda$,
will then contribute  to some $\SperX(k)$. (Note that $\SperX(1)=0$.)
Conversely any partition $\lambda'$ included in the sum $\SperX(k)$ for $0 \le k \le n$  is associated to a unique
partition of perimeter $n$ by adding $(n-k)$  parts of size $1$, since the perimeter increases
by $1$ when adding a new part of size $1$. We therefore have a norm-preserving bijection between
the two sets of partitions, and $\CperX(n) = \sum_{j=0}^N \SperX(j)$.  (Note that the 
contribution of  $1$ from the empty partition  $\SperX(0)$ must be included.)

(3) To show $\CperX(n) \le \CtotX(n)$ we observe that $\CtotX(n)$ is a sum  of reciprocal norms over all partitions with largest
part at most $n$, while $\CperX(n)$ is a sum over a subset of such partitions.  

(4) To show $\Ssize(n) =\CsizeX(n)$ is proved by an argument paralleling (2). We give a bijective size-preserving
bijection between the partitions included in the sum $\Ssize(n)$ and those in the union of the sets $\SsizeX(j)$
for $0 \le j \le n$. The bijection is the same map as in (2), removing all the parts $1$. Again $\SsizeX(1)=0$.
The remaining argument is the same as (2), each part $1$  removed changes the size by $1$.
\end{proof} 
 
%
%
\subsubsection{Proof of Theorem \ref{thm:W-per}: Perimeter-ensemble reciprocal norms}\label{subsec:32}

 \begin{proof} [Proof of Theorem \ref{thm:W-per}] 
Write  $\lambda= \lambda^{'} \cup \{1^{m_1}\}$ where $\lambda^{'}$ is a partition containing no
parts of size $1$, and  $m_1= {\mult_1}(\lambda)$ counts the number of its parts of size $1$.
Since parts of size $1$ do not affect the value  $N(\lambda)$, 
we  have $\frac{1}{N(\lambda)}= \frac{1}{N(\lambda')}$. 
We also have 
\begin{equation}
\label{eq:perimeter}
\per(\lambda) = \per(\lambda') + m_1,
\end{equation}
because  $\lambda$ and $\lambda_1$ have the same largest part, while the number of parts changes
by $m_1$.  In this formula  $\lambda' = \emptyset$ is exceptional; the formula \eqref{eq:perimeter} continues to work correctly 
by the convention  $\per(\emptyset)=1$.  We obtain
\begin{eqnarray*} 
 \Sper(n) &=& \sum_{\per(\lambda)=n} \frac{1}{N(\lambda)} \\
 &=& \sum_{m=0}^n \left(\sum_{ {\per(\lambda')=n-m}\atop{{\mult_1}(\lambda')=0}} \frac{1}{N(\lambda')}\right) \\
 &=& 
  \sum_{ {0 \le \per(\lambda') \le n}\atop{1 \not\in \lambda'}} \frac{1}{N(\lambda')} . 
 \end{eqnarray*} 
The  desired  upper bound  for $ \Sper(n)$ follows via  a telescoping product bound
 \begin{eqnarray*}
 \sum_{{ 0 \le \per(\lambda') \le n}\atop{1 \not\in \lambda'}}\frac{1}{N(\lambda')}  &\le& \prod_{k=2}^n( 1+ \frac{1}{k}+ \frac{1}{k^2} + \cdots)=  \prod_{k=2}^n (\frac{1}{1- \frac{1}{k}})  \\
 &=& \prod_{k=2}^n \frac{k}{k-1} = (\frac{2}{1})(\frac{3}{2}) \cdots (\frac{n}{n-1} ) =  n
 \end{eqnarray*}
 and  Theorem \ref{thm:W-per} follows. 
 \end{proof}
  
%
%
\subsubsection{ Proof of Theorem \ref{thm:W-tot}: Max-part ensemble reciprocal norms}\label{subsec:33}

 \begin{proof}[Proof of Theorem \ref{thm:W-tot}]
 (1) The bound comes from a telescoping product. We have $\StotX(0)=1$ (from the empty partition)
 and $\StotX(1)=0$. For $n \ge 2$, since there is at least one part of size $n$ we have
 \begin{align*}
 \StotX(n) = \frac{1}{n} \prod_{j =2}^n \left( 1+ \frac{1}{j} + \frac{1}{j^2} + \cdots \right) \\
 = \frac{1}{n} \prod_{j =2}^n \left(1- \frac{1}{n}\right)^{-1} = \frac{1}{n} \prod_{j=2}^{n} \left( \frac{j}{j-1} \right)= \frac{1}{n} \cdot n= 1.
 \end{align*}

 (2) The formula \eqref{eqn:cum-max-part-norm}  immediately follows from (1), noting that the sum
 $$
 \CtotX (n)= \sum_{j=0}^n \StotX(n).
 $$
 includes the empty partition.
 \end{proof}


\end{document}